\documentclass[a4paper,11pt,oneside]{article} 
\usepackage{amssymb}
\usepackage{color} \usepackage{mathrsfs} \let\mathcal\mathscr
\usepackage[all,ps,cmtip]{xy}
 \usepackage[top=1in, bottom=1in, left=1in, right=1in]{geometry}

\usepackage{tikz}
\usepackage{tikz-cd}

\title{\sc Hirzebruch signature theorem on Hochschild homology, with applications to derived invariance of Hodge numbers}
\author{\sc Roland Abuaf \footnote{Rectorat de Paris, 47 rue des \'Ecoles, 75005 Paris. \textit{rabuaf@gmail.com}}}
\date{}

\usepackage[T1]{fontenc}
\usepackage{amssymb}

\usepackage{amsmath}
\usepackage{latexsym}

\usepackage[latin1]{inputenc}
\usepackage{graphicx,txfonts}

\everymath{\displaystyle}
\newtheorem{theo}{Theorem}[section]
\newtheorem{theo*}{Theorem}
\newtheorem{conj*}{Conjecture}
\newtheorem{prob*}{Problem}
\newtheorem{prop*}{Proposition}
\newtheorem{quest*}{Question}

\newtheorem{rem}[theo]{Remark}

\newtheorem{prop}[theo]{Proposition}

\newtheorem{defi}[theo]{Definition}

\newtheorem{lem}[theo]{Lemma}
\newtheorem{cor}[theo]{Corollary}
\newtheorem{cor*}{Corollary}

\def\OM{\Omega}

\def\OO{\mathcal{O}}

\def\C{\mathbb{C}}
\def\R{\mathbb{R}}
\def\p{\varphi}

\newenvironment{proof}
{
\noindent
\textit{\underline{Proof}}:\\
$\blacktriangleright\;$%
}
{\hspace{\stretch{1}}%
$\blacktriangleleft$}

{
\noindent
\textit{\underline{Proof of Theorem #1}} :\\
$\blacktriangleright\;$%
}
{\hspace{\stretch{1}}%
$\blacktriangleleft$ \bigskip}

\begin{document}

\maketitle

\begin{abstract}
We prove a refinement of Hirzebruch's signature formula on the individual Hochschild diagonals of a smooth projective complex variety. Natural Hermitian forms, obtained by symmetrizing a Todd-corrected pairing, have explicitly computable positive and negative indices and are preserved by derived equivalences. On diagonals of the same parity as the dimension, their signatures separate the even and odd contributions to the Hochschild-Kostant-Rosenberg decomposition; on the other diagonals, their radicals and signatures are governed by the first Chern class. We deduce the derived invariance of $h^{n-2,0}$, $h^{n-1,1}$ and $h^{n-3,1}$ in any dimension, and of $h^{2,0}$ when the canonical bundle is trivial.

Combining the invariance of $h^{2,0}$ with the multiplicative structure of Hochschild cohomology and Verbitsky's orthogonal action, we give a new short proof of the theorem of Huybrechts and Nieper-Wi\ss kirchen that a derived partner of an irreducible holomorphic symplectic variety is again irreducible holomorphic symplectic. In dimension five with trivial canonical bundle, the signature invariants and the Libgober-Wood identity together establish the derived invariance of all Hodge numbers
\end{abstract}

\vspace{\stretch{1}}

\newpage

\section{Introduction}

The main result of this paper is a signature theorem on the individual Hochschild diagonals of a smooth projective complex variety. It refines Hirzebruch's signature formula by retaining the positive and negative indices on each diagonal, and provides new derived invariants of Hodge numbers in arbitrary dimension. Throughout the paper, varieties are connected, and derived equivalences are exact and $\C$-linear. Let $X$ be a smooth projective complex variety of dimension $n$. Put
\[
V_k(X)=\bigoplus_{p-q=k}H^{p,q}(X).
\]
The spaces $V_k(X)$ are the graded pieces of Hochschild homology, up to the sign convention for its grading. We consider the Hermitian forms
\[
\mathcal H_{k,X}(u,v)=\int_Xu^\vee\wedge\overline v\wedge
\frac{\operatorname{td}(X)+(-1)^{n+k}\operatorname{td}(X)^\vee}{2},
\]
for any $(u,v) \in V_k(X)^2$, with $u^\vee=i^j u$ for $u\in H^j(X,\C)$. They are defined without an auxiliary polarization. The following theorem combines their signature calculation with their derived invariance.

\begin{theo*} \label{theo1}
Let $X$ be a smooth projective complex variety of dimension $n$. If $k\equiv n\pmod2$, the form $\mathcal H_{k,X}$ is non-degenerate and
\[
\operatorname{sign}(\mathcal H_{k,X})=
\displaystyle\left(\sum_{\substack{p-q=k\\q\ \mathrm{even}}}h^{p,q}(X),\ \sum_{\substack{p-q=k\\q\ \mathrm{odd}}}h^{p,q}(X)\right), \ \textrm{if} \ n\equiv0,1\pmod4,
\]
and
\[
\operatorname{sign}(\mathcal H_{k,X})=
\displaystyle\left(\sum_{\substack{p-q=k\\q\ \mathrm{odd}}}h^{p,q}(X),\ \sum_{\substack{p-q=k\\q\ \mathrm{even}}}h^{p,q}(X)\right), \ \textrm{if} \ n\equiv2,3\pmod4.
\]
For every $k$, a derived equivalence $D^b(X)\simeq D^b(Y)$ induces an isometry
\[
(V_k(X),\mathcal H_{k,X})\simeq(V_k(Y),\mathcal H_{k,Y}).
\]
\end{theo*}

Multiplication by an invertible characteristic class reduces the calculation to an uncorrected pairing. Lefschetz decomposition and the Hodge-Riemann bilinear relations then determine its signature. In even complex dimension, summing the differences of indices with the appropriate sign recovers Hirzebruch's formula for the topological signature \cite[Theorem 15.8.2]{Hir66}. The theorem also gives Hermitian signatures in odd complex dimension. Derived invariance follows from adjointness for cohomological Fourier-Mukai transforms in the normalization supplied by Grothendieck-Riemann-Roch, as proved by C{\u{a}}ld{\u{a}}raru \cite{Cal05}. \\

The motivation for Theorem \ref{theo1} comes from the conjecture, usually attributed to Kontsevich and motivated by Homological Mirror Symmetry \cite{Kon95}, that derived-equivalent smooth projective complex varieties have the same Hodge numbers. Hochschild homology determines only the sums $\sum_{p-q=k}h^{p,q}$; together with Hodge symmetry and duality, these settle the conjecture in dimension at most two. Popa and Schnell's derived invariance of $h^{1,0}$ \cite{PS11} gives the threefold case. In \cite{Abu17}, I proved the derived invariance of every $h^{p,0}$ in dimension at most four. The argument uses the scalar action and trace for the image of a line bundle of non-zero rank; the existence of such an image in dimension four follows from the Hodge-Riemann bilinear relations. Consequently, in dimension four only $h^{1,1}$ and $h^{2,2}$ remained undetermined, with $2h^{1,1}+h^{2,2}$ already invariant. Other approaches, using generic vanishing and the Albanese morphism, give further invariance results for irregular varieties \cite{Popa13,LP15,CP19,CLP23}. In positive characteristic the conjecture fails, as shown by Addington and Bragg, with an appendix by Petrov \cite{AB23}.\\

The starting point of the present work was the additional relation supplied by the topological signature. As explained by Antieau in \cite{Ant26}, several members of an NSF FRG met at UC Berkeley in July 2026. During discussions about artificial intelligence in mathematics, they assembled problems related to their project and submitted them to Codex and Claude Code. Antieau reports that Codex produced a proof of the fourfold case; the decisive observation, attributed there to GPT-5.6 Sol, is that the topological signature is a derived invariant. Combined with \cite{Abu17}, this completes the fourfold case: the signature determines the remaining combination $h^{2,2}-2h^{1,1}$, independently of the Hochschild invariant $2h^{1,1}+h^{2,2}$.\\

Theorem \ref{theo1} above pushes further this observation by computing a separate signature on each Hochschild diagonal. Applied to $k=n-2$ and $k=n-4$, it gives the following consequences, proved in Corollary~\ref{corboundary}.

\begin{cor*}
Let $X$ and $Y$ be derived-equivalent smooth projective complex $n$-folds, with $n\geq2$. Then
\[
\begin{gathered}
h^{n-1,1}(X)=h^{n-1,1}(Y) \\
h^{n-2,0}(X) = h^{n-2,0}(Y) \\
h^{n-3,1}(X)=h^{n-3,1}(Y),\\
\end{gathered}
\]
If $K_X\simeq\mathcal O_X$, then $h^{2,0}(X)=h^{2,0}(Y)$.
\end{cor*}

This yields two applications. First, in dimension five the signatures separate further Hodge numbers. The remaining ambiguity disappears when the canonical bundle is trivial, by the Libgober-Wood identity \cite{LW90}.

\begin{theo*}\label{theomain}
Let $X$ and $Y$ be derived-equivalent smooth projective complex fivefolds. Then
\[
h^{p,q}(X)=h^{p,q}(Y)
\]
for $(p,q)\in\{(1,0),(3,0),(4,0),(5,0),(4,1),(2,1),(3,2)\}$, and
\[
h^{2,0}(X)+h^{3,1}(X)=h^{2,0}(Y)+h^{3,1}(Y) \quad \textrm{and} \quad
h^{1,1}(X)+h^{2,2}(X)=h^{1,1}(Y)+h^{2,2}(Y).
\]
Moreover, if $K_X\simeq\mathcal O_X$, then $h^{p,q}(X)=h^{p,q}(Y)$ for all $p,q$.
\end{theo*}

Second, we give a new proof of the following theorem.

\begin{theo*}[Huybrechts-Nieper-Wi\ss kirchen \cite{HNW11}]
Let $X$ and $Y$ be smooth projective complex varieties with $D^b(X)\simeq D^b(Y)$. If $X$ is irreducible holomorphic symplectic, then so is $Y$.
\end{theo*}

Starting from the invariance of $h^{2,0}$, the proof uses the multiplicative HKR isomorphism to obtain a symplectic form on $Y$. Verbitsky's orthogonal action on the full cohomology ring then identifies the ranks of multiplication by powers of non-zero isotropic classes, giving $h^{0,p}(Y)=h^{0,p}(X)$ for every $p$. Finally, the Beauville-Bogomolov decomposition and the holomorphic Lefschetz theorem prove simple connectivity. Remark~\ref{remHNWgap} explains the additional justification needed at the corresponding step of the original argument.

Section \ref{secsignature} constructs the forms and proves the signature formula. Section \ref{secapplications} contains the applications: Subsection \ref{subsecHodge} establishes the numerical invariance results, including the fivefold theorem, and discusses the remaining numerical ambiguities in dimensions five and six. Subsection \ref{subsecIHS} then gives the new symplectic argument, using the invariance of $h^{2,0}$ established in Subsection \ref{subsecHodge}. \\

\noindent \textbf{Acknowledgements and AI disclosure.} During the preparation of this article, the author used ChatGPT (OpenAI) as an interactive tool for mathematical discussion, verification of computations, exploration of examples and counter-examples, assistance in locating relevant references, and improvement of the exposition and English. The paper was written by the author, who is responsible for its correctness. \\

\section{The Hirzebruch-Hochschild signature theorem for projective varieties}\label{secsignature}

Throughout this section, $X$ is a smooth projective complex variety of dimension $n$. We regard $H^{p,q}(X)$ as a subspace of $H^{p+q}(X,\C)$ and use the usual complex conjugation on singular cohomology. For a possibly degenerate Hermitian form $\mathcal H$, we write $\operatorname{sign}(\mathcal H)=(n_+(\mathcal H),n_-(\mathcal H))$ for the ordered pair consisting of its positive and negative indices, and $n_0(\mathcal H)=\dim\operatorname{rad}(\mathcal H)$ for its nullity.

\begin{defi}\label{defHHform}
Let $X$ be a smooth projective variety of dimension $n$. For $-n\leq k\leq n$, put $V_k(X)=\bigoplus_{p-q=k}H^{p,q}(X)$. Define $u^\vee=i^ju$ for $u\in H^j(X,\C)$, extending this operation linearly to $H^\bullet(X,\C)$. The Todd-corrected cohomological pairing is:
\[
\mathcal T_{k,X}(u,v) =\int_Xu^\vee\wedge \overline{v}\wedge\operatorname{td}(X),
\]
For $u,v\in V_k(X)$, define
\[
\mathcal H_{k,X}(u,v)=\int_Xu^\vee\wedge\overline v\wedge \frac{\operatorname{td}(X)+(-1)^{n+k}\operatorname{td}(X)^\vee}{2}.
\]
\end{defi}

\begin{rem}\label{remHKR}
The Hochschild-Kostant-Rosenberg isomorphism identifies
\[
HH_k(X)\simeq\bigoplus_{q-p=k}H^q(X,\OM_X^p)=V_{-k}(X),
\]
so that the spaces $V_k(X)$ are, up to the sign of the grading, the graded pieces of
the Hochschild homology of $X$; we call them the Hochschild diagonals. Note that
$\mathcal H_{k,X}$ depends on no auxiliary choice: the ample class used in the proof of
Theorem \ref{theoHHsignature} below only serves to produce an adapted orthogonal
decomposition of $V_k(X)$.
\end{rem}

The Todd-corrected pairing is not Hermitian; its Hermitian part is what we shall use.
The following proposition computes it, and shows that the answer depends only on the
parity of $n+k$.

\begin{prop}\label{propHHparity}
The form $\mathcal H_{k,X}$ is Hermitian for every $k$, linear in its first variable and conjugate-linear in its second. More precisely,
\[
\mathcal H_{k,X}(u,v)=\frac12\left(\mathcal T_{k,X}(u,v)+\overline{\mathcal T_{k,X}(v,u)}\right).
\]
It is non-degenerate when $k\equiv n\pmod2$. When $k\not\equiv n\pmod2$, its radical is
\[
\operatorname{rad}(\mathcal H_{k,X})=\ker\bigl(c_1(X)\cup-:V_k(X)\longrightarrow V_k(X)\bigr).
\]
In particular, in the latter case it vanishes if $c_1(X)=0$, and it is degenerate whenever $V_k(X)\neq0$.
\end{prop}
\begin{proof}
The operation $u\mapsto u^\vee$ is $\C$-linear, whereas complex conjugation is $\C$-antilinear. Since the cup product and integration are $\C$-linear, the forms in Definition~\ref{defHHform} are sesquilinear.

We next prove the symmetry identity. By sesquilinearity, it suffices to consider $u\in H^{p,q}(X)$ and $v\in H^{p',q'}(X)$, where $p-q=p'-q'=k$. Write $\operatorname{td}_j(X)\in H^{j,j}(X)$ for the component of the Todd class of cohomological degree $2j$. The integral of $u\wedge\overline v\wedge\operatorname{td}_j(X)$ vanishes unless
\[
p+q'+j=q+p'+j=n.
\]
When these equalities hold, $p+q+p'+q'=2n-2j$. Integration commutes with complex conjugation, the Todd classes are real, and graded commutativity gives
\[
\begin{aligned}
\overline{i^{p'+q'}\int_Xv\wedge\overline u\wedge\operatorname{td}_j(X)}
&=(-i)^{p'+q'}(-1)^k\int_Xu\wedge\overline v\wedge\operatorname{td}_j(X)\\
&=(-1)^{n+k+j}i^{p+q}\int_Xu\wedge\overline v\wedge\operatorname{td}_j(X).
\end{aligned}
\]
Since $\operatorname{td}_j(X)^\vee=(-1)^j\operatorname{td}_j(X)$, summing over $j$ proves
\[
\overline{\mathcal T_{k,X}(v,u)}=(-1)^{n+k}\int_Xu^\vee\wedge\overline v\wedge\operatorname{td}(X)^\vee.
\]
The formula for $\mathcal H_{k,X}$ follows, and proves that it is Hermitian. Introduce the auxiliary form
\[
\mathcal H^0_{k,X}(u,v)=\int_Xu^\vee\wedge\overline v,\quad \textrm{forall} \ u,v \in V_k(X).
\]
We note that
\[\forall u,v \in V_k(X), \ \mathcal{H}_k(u,v) = \mathcal{H}^0_{k,X}\left(u, v \wedge \frac{\operatorname{td}(X)+(-1)^{n+k}\operatorname{td}(X)^\vee}{2} \right).\]

Let us check that $\mathcal{H}^0_{k,X}$ is non-degenerate : for any $0\neq u\in V_k(X)$, choose a non-zero Hodge component $u^{p,q}$. Poincar\'e duality and the Hodge decomposition imply that
\begin{equation*}
\begin{split}
 & H^{p,q}(X)\times H^{n-p,n-q}(X)\longrightarrow\C, \\
& (\alpha,\beta)\longmapsto\int_X\alpha\wedge\beta
\end{split}
\end{equation*}
is a perfect pairing. Choose $w\in H^{n-p,n-q}(X)$ with $\int_Xu^{p,q}\wedge w\neq0$ and put $v=\overline w\in H^{n-q,n-p}(X)\subset V_k(X)$. Every other Hodge component of $u$ pairs trivially with $w$, hence $\mathcal{H}^0_{k,X}(u,v) \neq0$. As a consequence, $\mathcal{H}_{k,X}^{0}$ is non degenerate. 

The identity $\operatorname{td}(X)^\vee=e^{-c_1(X)}\operatorname{td}(X)$ gives
\[
\frac{\operatorname{td}(X)+(-1)^{n+k}\operatorname{td}(X)^\vee}{2}
=\operatorname{td}(X)\frac{1+(-1)^{n+k}e^{-c_1(X)}}{2}.
\]
If $k\equiv n\pmod2$, the class $\frac{\operatorname{td}(X)+(-1)^{n+k}\operatorname{td}(X)^\vee}{2}$ has constant term $1$ and is therefore invertible. We deduce in that case that $\mathcal{H}_{k,X}$ is non-degenerate. \\

If $k\not\equiv n\pmod2$, the exponential expansion gives
\[
\frac{\operatorname{td}(X)-\operatorname{td}(X)^\vee}{2}
=\frac{c_1(X)}2 \operatorname{td}(X)\sum_{j\geq0}\frac{(-c_1(X))^j}{(j+1)!}.
\]
The class $ \operatorname{td}(X)\sum_{j\geq0}\frac{(-c_1(X))^j}{(j+1)!}$ has constant term $1$, hence is invertible. We deduce that the radical of $\mathcal H_{k,X}$ identifies with
\[
\ker\bigl(c_1(X)\cup-:V_k(X)\longrightarrow V_k(X)\bigr).
\]
In particular, the form vanishes if $c_1(X)=0$, and is degenerate whenever $V_k(X)\neq0$, since multiplication by $c_1(X)$ is nilpotent.

\end{proof}

\begin{theo}[Signature Theorem on Hochschild homology] \label{theoHHsignature}
Let $X$ be a smooth complex projective variety of dimension $n$. For every $k$ with $-n\leq k\leq n$ and $k\equiv n\pmod 2$, one has
\[
\operatorname{sign}(\mathcal H_{k,X})=
\displaystyle\left(\sum_{\substack{p-q=k\\q\ \mathrm{even}}}h^{p,q}(X),\ \sum_{\substack{p-q=k\\q\ \mathrm{odd}}}h^{p,q}(X)\right), \ \textrm{if} \ n\equiv0,1\pmod4,\]
and
\[
\operatorname{sign}(\mathcal H_{k,X})= \displaystyle\left(\sum_{\substack{p-q=k\\q\ \mathrm{odd}}}h^{p,q}(X),\ \sum_{\substack{p-q=k\\ q\ \mathrm{even}}}h^{p,q}(X)\right), \ \textrm{if} \ n\equiv2,3\pmod4.
\]
Equivalently,
\[
n_+(\mathcal H_{k,X})-n_-(\mathcal H_{k,X})=(-1)^{n(n-1)/2}\sum_{p-q=k}(-1)^qh^{p,q}(X).
\]
\end{theo}

\begin{proof}
Since $k\equiv n\pmod2$, put
\[
A_X=\sqrt{\frac{\operatorname{td}(X)+\operatorname{td}(X)^\vee}{2}},
\]
where the square root has constant term $1$. Cohomological duality acts on $H^{2j}(X,\C)$ by $(-1)^j$, so the average inside the square root has components only in degrees divisible by $4$. Its square root, obtained by the binomial expansion, has the same property, since products preserve this divisibility. Thus $A_X$ is real and invertible, and $A_X^\vee=A_X$. Moreover, its components have diagonal Hodge type, so multiplication by $A_X$ preserves each $V_k(X)$.

Since $A_X^\vee=A_X=\overline{A_X}$, we have
\[
\mathcal H_{k,X}(u,v)=\mathcal H^0_{k,X}(A_Xu,A_Xv).
\]
Multiplication by $A_X$ therefore identifies the two forms, so it is enough to compute the signature of $\mathcal H^0_{k,X}$. \\

Let $h$ be an ample class and put $L=h\cup-$. Since $h$ has Hodge type $(1,1)$, the operator $L$ maps $H^{p,q}(X)$ to $H^{p+1,q+1}(X)$. For $p+q\leq n$, denote by
\[
P^{p,q}(X)=\ker\bigl(L^{n-p-q+1}:H^{p,q}(X)\longrightarrow H^{n-q+1,n-p+1}(X)\bigr)
\]
the primitive cohomology. The Lefschetz decomposition respects the Hodge decomposition and gives, for every $(a,b)$,
\[
H^{a,b}(X)=\bigoplus_{r=\max(0,a+b-n)}^{\min(a,b)}L^rP^{a-r,b-r}(X).
\]
Summing over $a-b=k$ and setting $p=a-r$, $q=b-r$, we obtain $p-q=k$. The bounds on $r$ are equivalent to $p,q\geq0$, $p+q\leq n$ and $0\leq r\leq n-p-q$. Thus, regrouping the summands according to their primitive bidegrees yields
\[
V_k(X)=\bigoplus_{\substack{p-q=k\\p+q\leq n}}\ \bigoplus_{r=0}^{n-p-q}L^rP^{p,q}(X).
\]
In this expression, each primitive space $P^{p,q}(X)$ contributes its entire Lefschetz chain, whose terms have bidegrees $(p+r,q+r)$ and therefore all belong to $V_k(X)$.\\

For each primitive bidegree $(p,q)$ occurring in this decomposition, the Hodge-Riemann bilinear relations assert that the Hermitian form
\[
Q_{p,q}(\alpha,\beta)=(-1)^{(p+q)(p+q-1)/2}i^{p-q}\int_Xh^{n-p-q}\wedge\alpha\wedge\overline\beta
\]
is positive definite on $P^{p,q}(X)$. Choose a $Q_{p,q}$-orthonormal basis $(\alpha_{p,q,j})_{1\leq j\leq\dim P^{p,q}(X)}$ and consider the subspaces generated by the corresponding Lefschetz chains:
\[
W_{p,q,j}=\left\langle\alpha_{p,q,j},L\alpha_{p,q,j},\ldots,L^{n-p-q}\alpha_{p,q,j}\right\rangle.
\]
Hard Lefschetz implies that these generators are non-zero; since they have distinct total degrees, they are linearly independent. Thus
\[
V_k(X)=\bigoplus_{\substack{p-q=k\\p+q\leq n}}\ \bigoplus_jW_{p,q,j}.
\]

We first check that this decomposition is orthogonal for $\mathcal H^0_{k,X}$. Let $\alpha\in P^{p,q}(X)$ and $\beta\in P^{p',q'}(X)$, with $p-q=p'-q'=k$. By bidegree considerations, $\mathcal H^0_{k,X}(L^r\alpha,L^s\beta)$ vanishes unless
\[
r+s=n-\frac{p+q+p'+q'}2.
\]
When this equality holds,
\[
\mathcal H^0_{k,X}(L^r\alpha,L^s\beta)=i^{p+q+2r}\int_Xh^{r+s}\wedge\alpha\wedge\overline\beta.
\]
If $p'+q'>p+q$, the difference of these total degrees is at least $2$, so
\[
r+s=n-p'-q'+\frac{p'+q'-p-q}{2}\geq n-p'-q'+1.
\]
Consequently $L^{r+s}\overline\beta=0$ by primitivity, and the pairing vanishes. The case $p+q>p'+q'$ follows by Hermitian symmetry. Finally, if the total degrees agree, then $(p,q)=(p',q')$, and the only potentially non-zero pairings are scalar multiples of $Q_{p,q}(\alpha,\beta)$. The orthogonality of the chosen basis therefore proves that distinct subspaces $W_{p,q,j}$ are mutually orthogonal.\\

We now compute the signature on one such subspace. Fix $(p,q,j)$ and write
\[
\alpha=\alpha_{p,q,j},\qquad d=p+q,\qquad m=\frac{n-p-q}{2}.
\]
Here $m$ is a non-negative integer because $d\equiv p-q=k\equiv n\pmod2$. Since $Q_{p,q}(\alpha,\alpha)=1$, we have
\[
\mathcal H^0_{k,X}(L^r\alpha,L^s\alpha)=
\begin{cases}
(-1)^{q+r+d(d-1)/2},&r+s=2m,\\
0,&r+s\neq2m.
\end{cases}
\]
Indeed, in the first case the pairing is $i^{d+2r}\int_Xh^{n-d}\wedge\alpha\wedge\overline\alpha$, and division by the coefficient defining $Q_{p,q}$ gives
\[
\frac{i^{d+2r}}{(-1)^{d(d-1)/2}i^{p-q}}=(-1)^{q+r+d(d-1)/2}.
\]
For $0\leq r<m$, put
\[
U_r=\langle L^r\alpha,L^{2m-r}\alpha\rangle.
\]
Both diagonal entries vanish. The off-diagonal entries are $(-1)^{q+r+d(d-1)/2}$. Its matrix in the displayed basis is therefore
\[
\begin{pmatrix}
0&(-1)^{q+r+d(d-1)/2}\\
(-1)^{q+r+d(d-1)/2}&0
\end{pmatrix},
\]
which has eigenvalues $1$ and $-1$. Thus each $U_r$ is a hyperbolic plane of signature $(1,1)$. The vanishing condition $r+s=2m$ also shows that these planes and the central line $\langle L^m\alpha\rangle$ are pairwise orthogonal. Hence
\[
W_{p,q,j}=\left(\bigoplus_{r=0}^{m-1}U_r\right)\oplus\langle L^m\alpha\rangle
\]
is an orthogonal decomposition, with the first sum understood to be zero when $m=0$.\\

On the central line, the preceding formula gives
\[
\mathcal H^0_{k,X}(L^m\alpha,L^m\alpha)=(-1)^{q+m+d(d-1)/2}=(-1)^{q+n(n-1)/2}.
\]
The last equality follows from $n=d+2m$, since
\[
\frac{n(n-1)}2-\frac{d(d-1)}2=m(2d+2m-1)\equiv m\pmod2.
\]
Consequently,
\[
\operatorname{sign}\bigl(\mathcal H^0_{k,X}|_{W_{p,q,j}}\bigr)=
\begin{cases}
(m+1,m),& \textrm{if} \ q+n(n-1)/2\ \mathrm{even},\\
(m,m+1),& \ \textrm{otherwise}.
\end{cases}
\]
\smallskip

We now express the resulting signature in terms of Hodge numbers. Fix a primitive bidegree $(p,q)$ : there are exactly $\dim P^{p,q}(X)$ subspaces $W_{p,q,j}$ of this primitive type. 

Among the integers $q,q+1,\ldots,q+2m$, the numbers of even and odd integers are respectively $(m+1,m)$ if $q$ is even, and $(m,m+1)$ if $q$ is odd. Comparing with the signature computed above, we see that, when $n\equiv0,1\pmod4$, each $W_{p,q,j}$ contributes one positive index for each even integer $q+r$ and one negative index for each odd integer $q+r$, with $0\leq r\leq n-p-q$.

Since the $W_{p,q,j}$ are mutually orthogonal, their indices add, and therefore
\[
n_+(\mathcal H^0_{k,X})=
\sum_{\substack{p-q=k\\p+q\leq n}}\ \sum_{\substack{0\leq r\leq n-p-q\\q+r\ \mathrm{even}}}\dim P^{p,q}(X).
\]
To evaluate this sum, regroup its terms according to the final bidegree $(a,b)=(p+r,q+r)$. The Lefschetz decomposition recalled at the beginning of the proof gives
\[
h^{a,b}(X)=\sum_{r=\max(0,a+b-n)}^{\min(a,b)}\dim P^{a-r,b-r}(X),
\]
because $L^r$ is injective on each primitive space occurring in that decomposition. Consequently,
\[
\begin{aligned}
n_+(\mathcal H^0_{k,X})
&=\sum_{\substack{a-b=k\\b\ \mathrm{even}}}\ \sum_{r=\max(0,a+b-n)}^{\min(a,b)}\dim P^{a-r,b-r}(X)\\
&=\sum_{\substack{a-b=k\\b\ \mathrm{even}}}h^{a,b}(X).
\end{aligned}
\]
The same argument, with odd second Hodge index, gives
\[
n_-(\mathcal H^0_{k,X})=\sum_{\substack{a-b=k\\b\ \mathrm{odd}}}h^{a,b}(X).
\]
Finally, when $n\equiv2,3\pmod4$, the positive and negative indices on every chain are interchanged, so the two resulting sums are interchanged as well. The congruence above proves the asserted signature formula for $\mathcal H_{k,X}$ when $k\equiv n\pmod2$.
\end{proof}

The opposite parity is not vacuous, even though the form is always degenerate there.
The same change of variables reduces it to cup product with $c_1(X)$, whose indices
we now determine.

\begin{prop}\label{propHHopposite}
Let $X$ be a smooth complex projective variety of dimension $n$, and let $-n\leq k\leq n$ with $k\not\equiv n\pmod2$. The positive and negative indices of $\mathcal H_{k,X}$ are obtained by adding
\[
\sum_{\substack{p-q=k\\p+q<n-1}}
\operatorname{rk}\bigl(c_1(X)\cup-:H^{p,q}(X)\longrightarrow H^{p+1,q+1}(X)\bigr)
\]
to the corresponding indices of the Hermitian form
\[
(u,v)\longmapsto\frac{i^{n-1}}2\int_Xc_1(X)\wedge u\wedge\overline v
\quad\textrm{on}\quad H^{(n-1+k)/2,(n-1-k)/2}(X).
\]
Moreover,
\[
n_0(\mathcal H_{k,X})=\dim\ker\bigl(c_1(X)\cup-:V_k(X)\longrightarrow V_k(X)\bigr).
\]
In particular, if $c_1(X)=0$, the form vanishes.
\end{prop}

\begin{proof}
Since $k\not\equiv n\pmod2$, we have $(-1)^{n+k}=-1$, so the definition of our Hermitian form becomes
\[
\mathcal H_{k,X}(u,v)
=\int_Xu^\vee\wedge\overline v\wedge
\frac{\operatorname{td}(X)-\operatorname{td}(X)^\vee}{2}.
\]
Put
\[
B_X=\sqrt{\operatorname{td}(X)\sum_{j\geq0}\frac{(-c_1(X))^j}{(j+1)!}},
\]
where the sum is finite in cohomology and the square root has constant term $1$. Under cohomological duality, the sum is multiplied by $e^{c_1(X)}$, whereas the Todd class is multiplied by $e^{-c_1(X)}$. The expression under the square root is therefore fixed by duality, and so is its square root. Thus $B_X^\vee=B_X=\overline{B_X}$.

Using $\operatorname{td}(X)^\vee=e^{-c_1(X)}\operatorname{td}(X)$, the exponential expansion gives
\[
\frac{c_1(X)}2 B_X^2
=\frac{1-e^{-c_1(X)}}2\operatorname{td}(X)
=\frac{\operatorname{td}(X)-\operatorname{td}(X)^\vee}{2}.
\]
This is exactly the correcting class in $\mathcal H_{k,X}$ for the parity under consideration. Consequently,
\[
\begin{aligned}
\mathcal H_{k,X}(u,v)
&=\frac12\int_Xc_1(X)\wedge u^\vee\wedge\overline v\wedge B_X^2\\
&=\frac12\int_Xc_1(X)\wedge(B_Xu)^\vee\wedge\overline{B_Xv},
\end{aligned}
\]
where the second equality uses $B_X^\vee=B_X=\overline{B_X}$. Since $B_X$ has constant term $1$ and diagonal Hodge type, multiplication by $B_X$ is an automorphism of $V_k(X)$. Thus $\mathcal H_{k,X}$ has the same signature as the Hermitian form
\[
(u,v)\longmapsto\frac12\int_Xc_1(X)\wedge u^\vee\wedge\overline v.
\]

By bidegree considerations, this form pairs $H^{p,q}(X)$ only with $H^{n-1-q,n-1-p}(X)$. When $p+q<n-1$, these are distinct summands, and the matrix on their direct sum has the form
\[
\begin{pmatrix}0&M\\ M^*&0\end{pmatrix}.
\]
By Poincar\'e duality, the rank of $M$ equals the rank of
\[
c_1(X)\cup-:H^{p,q}(X)\longrightarrow H^{p+1,q+1}(X),
\]
since the target pairs perfectly with the complex conjugate of the complementary summand. A matrix of the displayed form has $\operatorname{rk}(M)$ positive and $\operatorname{rk}(M)$ negative eigenvalues, its remaining eigenvalues being zero.

These pairs of summands are mutually orthogonal. The only possible self-paired summand has total degree $n-1$ and is therefore $H^{(n-1+k)/2,(n-1-k)/2}(X)$. On this summand, $u^\vee=i^{n-1}u$, so the restriction is precisely the form appearing in the statement. Any summand with no complementary bidegree lies in the radical. Adding the contributions proves the formulas for the two indices, while the nullity follows from Proposition \ref{propHHparity}.
\end{proof}

\begin{rem} \label{remHirzebruch}
Suppose that $n$ is even. Hirzebruch's signature formula (see \cite[Theorem 15.8.2]{Hir66}) states that the difference between the positive and negative indices of the real symmetric intersection form
\begin{equation*}
\begin{split}
H^n(X,\R)\times H^n(X,\R) & \longrightarrow\R \\
(u,v)& \longmapsto\int_Xu\wedge v
\end{split}
\end{equation*}
is equal to
\[
\chi_1(X)=\sum_{p=0}^n\chi(X,\Omega_X^p)=\sum_{p,q}(-1)^qh^{p,q}(X).
\]
Our theorem refines this formula by determining the positive and negative indices separately on each Hochschild diagonal.

Indeed, the congruence in the proof of Theorem \ref{theoHHsignature} preserves every diagonal and identifies its signature with that of the auxiliary form. Consider this auxiliary Hermitian form $(u,v)\mapsto\int_Xu^\vee\wedge\overline v$ on $H^{\mathrm{even}}(X,\C)$. Its decomposition into the subspaces $V_k(X)$ with even $k$ is orthogonal. Furthermore, grouping complementary total degrees, each summand $H^j(X,\C)\oplus H^{2n-j}(X,\C)$ with even $j<n$ is hyperbolic and contributes zero to the difference of the indices. On the remaining summand $H^n(X,\C)$, the form is
\[
(u,v)\longmapsto(-1)^{n/2}\int_Xu\wedge\overline v.
\]
Thus the difference between the positive and negative indices of the real intersection form is
\[
(-1)^{n/2}\sum_{\substack{-n\leq k\leq n\\k\ \mathrm{even}}}\bigl(n_+(\mathcal H_{k,X})-n_-(\mathcal H_{k,X})\bigr).
\]
By Theorem~\ref{theoHHsignature}, since $n(n-1)/2\equiv n/2\pmod2$, this equals
\[
\sum_{\substack{p,q\\p-q\ \mathrm{even}}}(-1)^qh^{p,q}(X)=\sum_{p,q}(-1)^qh^{p,q}(X)=\chi_1(X).
\]
The omitted terms with $p-q$ odd cancel in pairs by Hodge symmetry, because $p$ and $q$ then have opposite parity. This recovers Hirzebruch's signature formula.

The classical formula therefore records a single signed sum of the differences of our indices, whereas Theorem \ref{theoHHsignature} retains the individual signatures on each diagonal. Moreover, our theorem also applies in odd complex dimension, where it gives Hermitian signatures on the odd Hochschild diagonals, although the real middle-dimensional intersection form is alternating.
\end{rem}

\section{Applications to derived invariance}\label{secapplications}

We now apply the signature theorem to derived invariance, first obtaining numerical consequences for Hodge numbers and then giving a new proof of the Huybrechts-Nieper-Wi\ss kirchen theorem.

\subsection{Signature invariants and Hodge numbers}\label{subsecHodge}
We first prove that derived equivalences preserve the Hermitian forms constructed above, and then use their signatures to establish the derived invariance of individual Hodge numbers.

\begin{theo}\label{theoHHinvariance}
Let $X$ and $Y$ be smooth projective complex $n$-folds with equivalent bounded derived categories. For every $k$, the Hermitian spaces
\[
\bigl(V_k(X),\mathcal H_{k,X}\bigr)
\quad\textrm{and}\quad
\bigl(V_k(Y),\mathcal H_{k,Y}\bigr)
\]
are isometric. In particular, the positive and negative indices and the nullity of $\mathcal H_{k,X}$ are derived invariants.
\end{theo}

\begin{proof}
We essentially follow C\u{a}ld\u{a}raru's proof \cite{Cal05} that the cohomological transform induced by a derived equivalence is an isometry for the Mukai pairing. \\

Represent a derived equivalence $\Phi:D^b(X)\simeq D^b(Y)$ by a Fourier-Mukai kernel $\mathcal P\in D^b(X\times Y)$, and denote the projections by $p_X,p_Y$. We use the following normalization of its cohomological transform:
\[
\Phi^H(u)=p_{Y*}\bigl(p_X^*(u\wedge\operatorname{td}(X))\wedge\operatorname{ch}(\mathcal P)\bigr).
\]
By Grothendieck-Riemann-Roch, this construction respects composition and sends the identity functor to the identity map. Thus, if $\Psi$ is a quasi-inverse of $\Phi$, the transformation $\Psi^H$ defined with the same normalization is the inverse of $\Phi^H$.

The Chern character and Todd classes are rational and have components of type $(r,r)$. Consequently, $\Phi^H$ commutes with complex conjugation and preserves $p-q$, hence induces an isomorphism $V_k(X)\simeq V_k(Y)$ for every $k$.

Since $\Psi$ is also the right adjoint of $\Phi$, its kernel is $\mathcal P^\vee\otimes p_X^*K_X[n]$, with the factors interchanged. Using
\[
(p_{Y*}w)^\vee=(-1)^np_{Y*}(w^\vee) \quad \textrm{and} \quad
\operatorname{td}(X)^\vee=e^{-c_1(X)}\operatorname{td}(X),
\]
the projection formula gives, for $u\in H^\bullet(X,\C)$ and $v\in H^\bullet(Y,\C)$,
\[
\begin{aligned}
\int_Y\Phi^H(u)^\vee\wedge v\wedge\operatorname{td}(Y)
&=(-1)^n\int_{X\times Y}
p_X^*\bigl(u^\vee\wedge\operatorname{td}(X)^\vee\bigr)
\wedge\operatorname{ch}(\mathcal P)^\vee
\wedge p_Y^*\bigl(v\wedge\operatorname{td}(Y)\bigr)\\
&=\int_Xu^\vee\wedge\Psi^H(v)\wedge\operatorname{td}(X).
\end{aligned}
\]
Indeed, $\operatorname{ch}(\mathcal P)^\vee=\operatorname{ch}(\mathcal P^\vee)$, and the factor $(-1)^ne^{-c_1(X)}$ is precisely the contribution of $K_X[n]$ to the Chern character of the adjoint kernel.

Since $\Psi^H\circ\Phi^H=\mathrm{id}$ and $\Phi^H$ commutes with complex conjugation, it follows that, for $u,v\in V_k(X)$,
\[
\int_Y\Phi^H(u)^\vee\wedge\overline{\Phi^H(v)}\wedge\operatorname{td}(Y)
=\int_Xu^\vee\wedge\overline v\wedge\operatorname{td}(X).
\]
The preceding integral identity gives
\[
\mathcal T_{k,Y}\bigl(\Phi^H(u),\Phi^H(v)\bigr)=\mathcal T_{k,X}(u,v).
\]
Applying this equality also with $u,v$ interchanged, Proposition \ref{propHHparity} yields
\[
\begin{aligned}
\mathcal H_{k,Y}\bigl(\Phi^H(u),\Phi^H(v)\bigr)
&=\frac12\left(
\mathcal T_{k,Y}\bigl(\Phi^H(u),\Phi^H(v)\bigr)
+\overline{\mathcal T_{k,Y}\bigl(\Phi^H(v),\Phi^H(u)\bigr)}
\right)\\
&=\frac12\left(\mathcal T_{k,X}(u,v)+\overline{\mathcal T_{k,X}(v,u)}\right)\\
&=\mathcal H_{k,X}(u,v).
\end{aligned}
\]
Since $\Phi^H:V_k(X)\to V_k(Y)$ is an isomorphism, this proves the required isometry. 
\end{proof}
\begin{cor}\label{corboundary}
Let $X$ and $Y$ be derived-equivalent smooth projective complex $n$-folds. For every $k\equiv n\pmod2$, the two sums
\[
\sum_{\substack{p-q=k\\q\ \mathrm{even}}}h^{p,q}(X),\quad
\sum_{\substack{p-q=k\\q\ \mathrm{odd}}}h^{p,q}(X)
\]
are equal to the corresponding sums for $Y$. In particular, for $n\geq2$,
\[
h^{n-2,0}(X)=h^{n-2,0}(Y) \quad \textrm{and} \quad h^{n-1,1}(X)=h^{n-1,1}(Y).
\]
The numbers $h^{n,0}$ and $h^{n-1,0}$ are also derived invariants. For $n\geq4$, we further have
\[
\begin{gathered}
h^{n-3,1}(X)=h^{n-3,1}(Y),\\
2h^{n-4,0}(X)+h^{n-2,2}(X)=2h^{n-4,0}(Y)+h^{n-2,2}(Y).
\end{gathered}
\]
Consequently, $h^{n-2,2}(X)=h^{n-2,2}(Y)$ whenever $h^{n-4,0}(X)=h^{n-4,0}(Y)$. If in addition $K_X\simeq\mathcal O_X$, then
\[
h^{2,0}(X)=h^{2,0}(Y).
\]
\end{cor}
\begin{proof}
The first assertion follows from Theorems~\ref{theoHHsignature} and~\ref{theoHHinvariance}. Applied to $V_{n-2}(X)=H^{n-2,0}(X)\oplus H^{n-1,1}(X)\oplus H^{n,2}(X)$, the signature formula gives
\[
\operatorname{sign}(\mathcal H_{n-2,X})=\begin{cases}
\bigl(2h^{n-2,0}(X),h^{n-1,1}(X)\bigr),&n\equiv0,1\pmod4,\\
\bigl(h^{n-1,1}(X),2h^{n-2,0}(X)\bigr),&n\equiv2,3\pmod4,
\end{cases}
\]
since $h^{n,2}(X)=h^{n-2,0}(X)$ by Hodge symmetry and Poincar\'e duality, without any hypothesis on $K_X$. Both entries are therefore derived invariants. The HKR equalities $\dim V_n=h^{n,0}$ and $\dim V_{n-1}=2h^{n-1,0}$ prove the assertions concerning these boundary numbers.

For $n\geq4$, applying the signature formula to $V_{n-4}(X)$ similarly gives
\[
\operatorname{sign}(\mathcal H_{n-4,X})=\begin{cases}
\bigl(2h^{n-4,0}(X)+h^{n-2,2}(X),\,2h^{n-3,1}(X)\bigr),&n\equiv0,1\pmod4,\\
\bigl(2h^{n-3,1}(X),\,2h^{n-4,0}(X)+h^{n-2,2}(X)\bigr),&n\equiv2,3\pmod4.
\end{cases}
\]
Here Hodge symmetry and Poincar\'e duality identify $h^{n,4}$ with $h^{n-4,0}$ and $h^{n-1,3}$ with $h^{n-3,1}$. This proves the additional assertions.

Finally, if $K_X\simeq\mathcal O_X$, uniqueness of the Serre functor \cite{BK89} gives $\Phi\circ S_X\simeq S_Y\circ\Phi$. Since $S_X\simeq[n]$, it follows that $S_Y\simeq[n]$, and evaluation at $\mathcal O_Y$ gives $K_Y\simeq\mathcal O_Y$. Hodge symmetry and Serre duality then yield
\[
h^{n-2,0}(X)=h^{0,n-2}(X)=h^{0,2}(X)=h^{2,0}(X),
\]
and the same equalities hold for $Y$. This proves the final assertion.
\end{proof}

\begin{theo}\label{theofivefold}
Let $X$ and $Y$ be derived-equivalent smooth projective complex fivefolds. Then
\[
h^{p,q}(X)=h^{p,q}(Y)
\]
for $(p,q)\in\{(1,0),(3,0),(4,0),(5,0),(4,1),(2,1),(3,2)\}$, and
\[
h^{2,0}(X)+h^{3,1}(X)=h^{2,0}(Y)+h^{3,1}(Y),\qquad
h^{1,1}(X)+h^{2,2}(X)=h^{1,1}(Y)+h^{2,2}(Y).
\]
Moreover, if $K_X\simeq\mathcal O_X$, then $K_Y\simeq\mathcal O_Y$ and
\[
\forall (p,q) \in \{0,\ldots,5\}^2, \ h^{p,q}(X)=h^{p,q}(Y).
\]
\end{theo}
\begin{proof}
The invariance of Hochschild homology and the Hochschild-Kostant-Rosenberg isomorphism give
\[
\dim V_k(X)=\sum_{p-q=k}h^{p,q}(X)=\sum_{p-q=k}h^{p,q}(Y)=\dim V_k(Y)
\]
for every $k$. Popa-Schnell's theorem \cite{PS11} gives the invariance of $h^{1,0}$, while Corollary~\ref{corboundary} gives that of $h^{3,0}$ and $h^{4,1}$, as well as $h^{4,0}$ and $h^{5,0}$. Hodge symmetry and Poincar\'e duality give
\[
\dim V_2(X)=2h^{2,0}(X)+2h^{3,1}(X),
\]
so $h^{2,0}+h^{3,1}$ is invariant. Next, Theorem~\ref{theoHHsignature} gives
\[
\operatorname{sign}(\mathcal H_{1,X})=\bigl(2h^{1,0}(X)+h^{3,2}(X),\,2h^{2,1}(X)\bigr).
\]
By Theorem~\ref{theoHHinvariance} and the invariance of $h^{1,0}$, both $h^{2,1}$ and $h^{3,2}$ are therefore invariant. Finally,
\[
\dim V_0(X)=2+2h^{1,1}(X)+2h^{2,2}(X)
\]
shows that $h^{1,1}+h^{2,2}$ is invariant. This proves the first assertion without any assumption on the canonical bundle.\\

Moreover, suppose that $K_X\simeq\mathcal O_X$. By Corollary~\ref{corboundary}, $K_Y\simeq\mathcal O_Y$ and $h^{2,0}$ is invariant. Equivalently, Serre duality gives $h^{2,0}=h^{3,0}$ on both varieties. The sum $h^{2,0}+h^{3,1}$ then determines $h^{3,1}$. All boundary Hodge numbers are determined, since $h^{p,0}=h^{5-p,0}$ and $h^{0,0}=h^{5,0}=1$. Thus, by Hodge symmetry and Serre duality, only $h^{1,1}$ and $h^{2,2}$ remain to be separated. We conclude with the Libgober-Wood identity.\\

Now, the Libgober-Wood identity \cite{LW90}, valid for every smooth projective complex $n$-fold with $c_1=0$, gives
\[
\sum_{p,q}(-1)^{p+q}\left(\left(p-\frac n2\right)^2-\frac n{12}\right)h^{p,q}(X)=0.
\]
For $n=5$, multiplying by $6$ and summing first over $q$, this becomes
\[
35\chi(X,\mathcal O_X)-11\chi(X,\Omega_X^1)-\chi(X,\Omega_X^2)+\chi(X,\Omega_X^3)+11\chi(X,\Omega_X^4)-35\chi(X,\Omega_X^5)=0.
\]
Serre duality gives $\chi(X,\Omega_X^{5-p})=-\chi(X,\Omega_X^p)$, while $K_X\simeq\mathcal O_X$ implies $\chi(X,\mathcal O_X)=0$. Hence
\[
\chi(X,\Omega_X^2)=-11\chi(X,\Omega_X^1).
\]
Expanding these Euler characteristics, the boundary terms cancel by Hodge symmetry, Serre duality and the triviality of the canonical bundle, yielding
\[
11h^{1,1}(X)-10h^{1,2}(X)-h^{2,2}(X)+h^{2,3}(X)+10h^{1,3}(X)-11h^{1,4}(X)=0.
\]
Every term other than $h^{1,1}$ and $h^{2,2}$ is already invariant, so $11h^{1,1}-h^{2,2}$ is invariant. Together with the invariance of $h^{1,1}+h^{2,2}$, this determines both remaining Hodge numbers and proves the theorem.
\end{proof}

\begin{rem}
\begin{enumerate}
\item 
When $K_X \not\simeq \OO_X$, the signatures on the even Hochschild diagonals do not determine the remaining Hodge numbers, even when combined with the dimensions of Hochschild homology. An explicit example is provided by the following smooth projective rational fivefolds.

Let $S_r$ be the blow-up of $\mathbf P^2$ at $r$ points in general position, and set
\[
X=S_1\times\mathbf P_{S_7}\bigl(\mathcal O_{S_7}\oplus K_{S_7}^{-1}\bigr),
\qquad
Y=S_2\times\mathbf P_{S_5}\bigl(\mathcal O_{S_5}\oplus K_{S_5}^{-2}\bigr),
\]
where projectivization parametrizes one-dimensional quotients. The blow-up and projective bundle formulas show that all their off-diagonal Hodge numbers vanish, while
\[
\begin{array}{c|rrrrrr}
 &h^{0,0}&h^{1,1}&h^{2,2}&h^{3,3}&h^{4,4}&h^{5,5}\\ \hline
X&1&11&28&28&11&1\\
Y&1&10&29&29&10&1
\end{array}
\]
Consequently, $V_k(X)=V_k(Y)=0$ for $k\neq0$, and
\[
\dim V_0(X)=\dim V_0(Y)=80.
\]
Both first Chern classes are non-zero, since their restrictions to the fibres of the projective bundles have degree $2$. Nevertheless,
\[
\operatorname{sign}(\mathcal H_{0,X})
=\operatorname{sign}(\mathcal H_{0,Y})=(29,22),
\qquad
n_0(\mathcal H_{0,X})=n_0(\mathcal H_{0,Y})=29.
\]

These varieties thus have the same Hochschild dimensions and the same signatures and nullities for all the forms considered here, but different $h^{1,1}$ and $h^{2,2}$. Moreover, the Libgober-Wood identity gives
\[
\int_Xc_1(X)c_4(X)=256,\qquad
\int_Yc_1(Y)c_4(Y)=232,
\]
so these numerical invariants do not determine the Chern number appearing in that identity either. We nonetheless note that the varieties $X$ and $Y$ are not derived equivalent. Indeed, the projective bundle formula gives
\[
h^0(X,K_X^{-1})=99,\qquad
h^0(Y,K_Y^{-1})=240,
\]
whereas these dimensions are preserved by derived equivalence, by the derived invariance of the graded anticanonical ring. 

\item For smooth projective complex sixfolds with trivial canonical bundle, the preceding results also give a reduction of the derived invariance problem to two Hodge numbers. More precisely, the numbers
\[
h^{1,0},\quad h^{2,0},\quad h^{3,1},\quad h^{5,1},\quad h^{2,2},\quad h^{4,2}
\]
and the combinations
\[
h^{3,0}+h^{4,1},\qquad h^{2,1}+h^{3,2},\qquad 2h^{1,1}+h^{3,3},\qquad h^{1,1}-h^{3,0}-h^{2,1}
\]
are derived invariants.

Indeed, Theorem \ref{theoHHsignature}, together with Hodge symmetry, Serre duality and the triviality of the canonical bundle, gives
\[
\begin{aligned}
\operatorname{sign}(\mathcal H_{4,X})&=\bigl(h^{5,1}(X),\,2h^{2,0}(X)\bigr),\\
\operatorname{sign}(\mathcal H_{2,X})&=\bigl(2h^{3,1}(X),\,2h^{2,0}(X)+h^{4,2}(X)\bigr),\\
\operatorname{sign}(\mathcal H_{0,X})&=\bigl(2h^{1,1}(X)+h^{3,3}(X),\,2+2h^{2,2}(X)\bigr).
\end{aligned}
\]
Their derived invariance proves the assertions concerning the individual Hodge numbers, except for $h^{1,0}$, which follows from \cite{PS11}, and also proves the invariance of $2h^{1,1}+h^{3,3}$. Moreover, the HKR equalities
\[
\dim V_3(X)=2h^{3,0}(X)+2h^{4,1}(X) \quad \textrm{and} \quad
\dim V_1(X)=2h^{1,0}(X)+2h^{2,1}(X)+2h^{3,2}(X)
\]
give the invariance of $h^{3,0}+h^{4,1}$ and $h^{2,1}+h^{3,2}$.

Finally, the Libgober-Wood identity \cite{LW90} in dimension six, using $c_1(X)=0$ and $\chi(X,\Omega_X^{6-p})=\chi(X,\Omega_X^p)$, reads
\[
34\chi(X,\mathcal O_X)-14\chi(X,\Omega_X^1)+2\chi(X,\Omega_X^2)+\chi(X,\Omega_X^3)=0.
\]
Expanding the Euler characteristics gives
\[
\begin{aligned}
0={}&68-96h^{1,0}+72h^{2,0}-32h^{3,0}+14h^{1,1}-16h^{2,1}\\
&+12h^{3,1}-16h^{4,1}+14h^{5,1}+2h^{2,2}+2h^{4,2}-h^{3,3},
\end{aligned}
\]
where all Hodge numbers refer to $X$. The terms involving the Hodge numbers not yet individually determined can be grouped as
\[
-32h^{3,0}+14h^{1,1}-16h^{2,1}-16h^{4,1}-h^{3,3}
=16(h^{1,1}-h^{3,0}-h^{2,1})-16(h^{3,0}+h^{4,1})-(2h^{1,1}+h^{3,3}).
\]
Every other term and both parenthesized combinations on the right already shown above are derived invariant. Hence $h^{1,1}-h^{3,0}-h^{2,1}$ is invariant as well.

Consequently, if two such derived-equivalent sixfolds have equal $h^{3,0}$ and $h^{2,1}$, these invariant combinations successively determine $h^{4,1}$, $h^{3,2}$, $h^{1,1}$ and $h^{3,3}$. All remaining Hodge numbers follow from the invariants listed above, Hodge symmetry, Serre duality and the identities $h^{p,0}=h^{6-p,0}$. Thus proving the derived invariance of $h^{3,0}$ and $h^{2,1}$ would establish the full conjecture for sixfolds with trivial canonical bundle.
\end{enumerate}
\end{rem}

\subsection{A new proof of the Huybrechts-Nieper-Wi\ss kirchen theorem}\label{subsecIHS}
Recall that a compact K\"ahler manifold is \emph{irreducible holomorphic symplectic} if it is simply connected and its space of holomorphic two-forms is generated by an everywhere non-degenerate form. The derived invariance of $h^{2,0}$ established in the preceding subsection provides the starting point for a new proof that every smooth projective derived partner of an irreducible holomorphic symplectic variety is again irreducible holomorphic symplectic.

We begin with the consequence of Verbitsky's orthogonal action needed in the proof.

\begin{lem}\label{lemisotropicranks}
Let $X$ be an irreducible holomorphic symplectic variety of dimension $2n$, with Beauville-Bogomolov form $q_X$ and holomorphic symplectic form $\sigma_X$. For every non-zero $q_X$-isotropic class $\alpha\in H^2(X,\C)$, there exists a graded algebra automorphism $G$ of $H^\bullet(X,\C)$ such that $G(\alpha)=\sigma_X$. In particular, $\alpha^n\neq0$, and for every $j$,
\[
\operatorname{rk}\biggl(\alpha^n\cup-:H^j(X,\C)\to H^{j+2n}(X,\C)\biggr)
=h^{0,j}(X).
\]
\end{lem}
\begin{proof}
Verbitsky's action of $\mathfrak{so}(H^2(X,\C),q_X)$ on the full cohomology ring is by degree-preserving derivations and restricts to the standard representation on $H^2(X,\C)$ \cite{Ver95}. It integrates to an action of $\operatorname{Spin}(H^2(X,\C),q_X)$ by graded algebra automorphisms of $H^{\bullet}(X, \C)$; see also \cite[Proposition 6.1]{Bot22}. Since $b_2(X)\geq3$, the complex special orthogonal group is transitive on non-zero isotropic vectors. Both $\alpha$ and $\sigma_X$ are such vectors, so an orthogonal transformation taking $\alpha$ to $\sigma_X$ lifts to the required automorphism $G$. In particular, $G(\alpha^n)=\sigma_X^n\neq0$, and multiplicativity gives
\[
G\circ L_{\alpha^n}=L_{\sigma_X^n}\circ G,
\quad \textrm{where} \ L_\beta(u)=\beta\cup u.
\]
The two multiplication maps thus have the same rank in each total degree. \\

On $H^{a,b}(X)$, multiplication by $\sigma_X^n$ is zero if $a>0$, by Hodge type. On $H^{0,j}(X)$ it is an isomorphism onto $H^{2n,j}(X)$, induced by the trivialization $\mathcal O_X\xrightarrow{\sim}K_X$ given by $\sigma_X^n$. Hence
\begin{equation}\label{eqsymplectickernel}
\ker\bigl(L_{\sigma_X^n}|_{H^j(X,\C)}\bigr)
=\bigoplus_{\substack{a+b=j\\a>0}}H^{a,b}(X),
\end{equation}
which proves that
\[
\operatorname{rk}\biggl(\sigma_X^n\cup-:H^j(X,\C)\to H^{j+2n}(X,\C)\biggr)
=h^{0,j}(X).
\]
As a consequence, the relation $G\circ L_{\alpha^n}=L_{\sigma_X^n}\circ G$, with $G$ being a graded automorphism of the full cohomology ring, gives
\[
\operatorname{rk}\biggl(\alpha^n\cup-:H^j(X,\C)\to H^{j+2n}(X,\C)\biggr)
=h^{0,j}(X).
\]
\end{proof}

\begin{theo}[Huybrechts-Nieper-Wi\ss kirchen \cite{HNW11}]\label{theoIHS}
Let $X$ and $Y$ be smooth projective complex varieties with equivalent bounded derived categories. If $X$ is irreducible holomorphic symplectic, then $Y$ is irreducible holomorphic symplectic.
\end{theo}
\begin{proof}
Write $\dim X=2n$. Derived invariance of dimension and compatibility with Serre functors give $\dim Y=2n$ and $K_Y\simeq\mathcal O_Y$ \cite{Huy06,BK89}. Corollary~\ref{corboundary} gives
\[
h^{2,0}(Y)=h^{2,0}(X)=1.
\]
We prove successively that $Y$ is holomorphic symplectic, that its holomorphic Euler characteristic is $n+1$, and that it is simply connected.

\smallskip\noindent\textit{Existence of a symplectic form on $Y$.}
For a smooth projective variety $Z$, set
\[
HT^j(Z)=\bigoplus_{a+b=j}H^b(Z,\bigwedge^aT_Z).
\]
The HKR isomorphism is an isomorphism of graded algebras $HH^\bullet(Z)\simeq HT^\bullet(Z)$ \cite[Corollary 1.5]{CVdB10}. On $X$, contraction with $\sigma_X$ gives $T_X\simeq\Omega_X^1$ and hence a multiplicative identification
\[
HT^\bullet(X)\simeq\bigoplus_{a,b}H^b(X,\Omega_X^a)
\simeq H^\bullet(X,\C),
\]
with total degree $a+b$ and the usual graded cup products. Derived invariance of the Hochschild cohomology algebra therefore yields a graded algebra isomorphism
\[
\Theta:HT^\bullet(Y)\xrightarrow{\sim}H^\bullet(X,\C).
\]
Choose $0\neq\eta\in H^2(Y,\mathcal O_Y)\subset HT^2(Y)$ and put $\alpha=\Theta(\eta)$. Since $\eta^{n+1}\in H^{2n+2}(Y,\mathcal O_Y)=0$, we have $\alpha^{n+1}=0$, hence $\alpha^{2n}=0$. Fujiki's relation gives
\[
0=\int_X\alpha^{2n}=c_Xq_X(\alpha)^n,\quad \textrm{with} \ c_X\neq0.
\]
Thus $\alpha$ is non-zero and isotropic. Lemma \ref{lemisotropicranks} implies $\alpha^n\neq0$, and consequently $\eta^n\neq0$. Under the Dolbeault identification, the complex conjugate of $\eta$ is a holomorphic two-form $\sigma_Y$ with $\sigma_Y^n\neq0$. Since $K_Y\simeq\mathcal O_Y$ and $Y$ is connected and projective, this non-zero section of $K_Y$ is nowhere vanishing. Therefore $\sigma_Y$ is everywhere non-degenerate.

\smallskip\noindent\textit{Ranks of multiplication and holomorphic forms.}
Both varieties are now holomorphic symplectic. Using their symplectic forms to identify polyvectors with forms, the same multiplicative HKR argument gives a graded algebra isomorphism
\[
\Psi:H^\bullet(Y,\C)\xrightarrow{\sim}H^\bullet(X,\C).
\]
No compatibility with Hodge decompositions is asserted. Put $\beta=\Psi(\sigma_Y)$. As above, $\beta\neq0$ and $\beta^{n+1}=0$ imply $q_X(\beta)=0$. The algebra isomorphism $\Psi$ and Lemma \ref{lemisotropicranks} on $X$ give, for every $j$,
\[
\begin{aligned}
h^{0,j}(Y)
&=\operatorname{rk}\bigl(L_{\sigma_Y^n}:H^j(Y,\C)\to H^{j+2n}(Y,\C)\bigr)\\
&=\operatorname{rk}\bigl(L_{\beta^n}:H^j(X,\C)\to H^{j+2n}(X,\C)\bigr)\\
&=h^{0,j}(X).
\end{aligned}
\]
For an irreducible holomorphic symplectic variety, the holomorphic forms are generated by its symplectic form \cite{Bea83}. Consequently,
\begin{equation}\label{eqIHSholomorphicforms}
h^{0,j}(Y)=
\begin{cases}1,&j=0,2,\ldots,2n,\\0,&j\text{ odd},\end{cases}
\qquad \textrm{and} \quad \chi(Y,\mathcal O_Y)=n+1.
\end{equation}

\smallskip\noindent\textit{Simple connectivity.}
By the Beauville-Bogomolov decomposition, there is a connected finite \'etale cover $\pi:\widetilde Y\to Y$ with
\[
\widetilde Y\simeq A\times\prod_i M_i\times\prod_\ell C_\ell,
\]
where $A$ is a complex torus, the $M_i$ are irreducible holomorphic symplectic, and the $C_\ell$ are simply connected Calabi-Yau manifolds of dimension at least three, with no holomorphic forms in intermediate degrees. A factor $C_\ell$ has neither holomorphic one-forms nor holomorphic two-forms. Writing $\widetilde Y=C_\ell\times W$, the K\"unneth decomposition and the vanishings
$H^0(C_\ell,\Omega_{C_\ell}^1)=H^0(C_\ell,\Omega_{C_\ell}^2)=0$ show that every holomorphic two form on $\widetilde Y$ is pulled back from $W$. Its contraction with any vector tangent to
$C_\ell$ is therefore zero. Since $\pi^*\sigma_Y$ is everywhere non-degenerate, no such factor $C_\ell$ can occur.

If $A$ had positive dimension, multiplicativity of the holomorphic Euler characteristic on products would give $\chi(\widetilde Y,\mathcal O_{\widetilde Y})=0$. On the other hand, Hirzebruch-Riemann-Roch for the \'etale cover gives
\[
\chi(\widetilde Y,\mathcal O_{\widetilde Y})
=(\deg\pi)\chi(Y,\mathcal O_Y)=(\deg\pi)(n+1)\neq0.
\]
Hence $A$ is a point and $\widetilde Y=\prod_{i=1}^rM_i$ is simply connected. It is therefore the universal cover, with finite deck group $\Gamma$.

By uniqueness of the Beauville-Bogomolov decomposition \cite[Theorem 2]{Bea83}, every automorphism of $\widetilde Y=\prod_{i=1}^r M_i$ permutes the factors $M_i$. In particular, this holds for every deck transformation. Since $H^0(M_i,\Omega_{M_i}^1)=0$, we can write
\[
\pi^*\sigma_Y=\sum_{i=1}^r\tau_i,
\]
where each $\tau_i$ is the pullback of a non-zero symplectic form on $M_i$. Every element of $\Gamma$ preserves this sum.

Let $g\in\Gamma$ and consider one cycle of its permutation of the factors:
\[
f_1:M_{i_1}\to M_{i_2},\ \ldots,\ 
f_s:M_{i_s}\to M_{i_1}.
\]
The equality $g^*\pi^*\sigma_Y=\pi^*\sigma_Y$ gives $f_j^*\tau_{i_{j+1}}=\tau_{i_j}$, with cyclic indices. Thus $f=f_s\circ\cdots\circ f_1$ is a symplectic automorphism of $M_{i_1}$. If $\dim M_{i_1}=2m$, its action on each $H^{0,2p}(M_{i_1})=\C\overline{\tau}_{i_1}^{\wedge p}$ is the identity, and the odd groups vanish. Its holomorphic Lefschetz number is therefore
\[
\sum_q(-1)^q\operatorname{Tr}\bigl(f^*|H^q(M_{i_1},\mathcal O_{M_{i_1}})\bigr)=m+1\neq0.
\]
The holomorphic Lefschetz theorem \cite{AtB67} implies that $f$ has a fixed point. Starting from this point and applying $f_1,\ldots,f_{s-1}$ produces a tuple fixed by $g$ on the factors in the cycle. Doing this for every cycle gives a fixed point of $g$ on $\widetilde Y$. Non-identity deck transformations act freely, so $\Gamma$ is trivial. Hence $Y$ is simply connected; together with $h^{2,0}(Y)=1$ and the non-degeneracy of $\sigma_Y$, this proves the theorem.
\end{proof}

\begin{rem}\label{remHNWgap}
Theorem~\ref{theoIHS} is due to Huybrechts and Nieper-Wi\ss kirchen \cite[Theorem 4.9, pp.~954-955]{HNW11}. The proof above offers a somewhat shorter and simpler approach, providing an alternative to two steps in the published argument that require further justification. \\

The final application of Proposition A.1 in the proof of Huybrechts-Nieper-Wi{\ss}kirchen's Theorem 4.9 requires the full list of values of $h^{p,0}$, which is not explicitly established there. Once holomorphic symplecticity has been proved, their Lemma 4.8 yields $h^{2,0}=1$, while the invariance of $HH^1$ gives $h^{1,0}=0$. Schwald's subsequent characterization of irreducible holomorphic symplectic manifolds \cite[Theorem 1]{Schwald} shows that these two equalities suffice, thus providing a way to complete their argument without determining the remaining $h^{p,0}$ beforehand. Our proof provides an alternative : the rank calculation above, using Verbitsky's action on the full cohomology, directly establishes \eqref{eqIHSholomorphicforms}. In particular, it gives $\chi(\mathcal O_Y)=n+1$, which excludes a torus factor in a finite \'etale Beauville-Bogomolov cover and allows the geometric argument to conclude without invoking Schwald's theorem. \\

A separate, readily repairable point occurs at the beginning of the proof of Proposition A.1, on p.956: $h^{1,0}=0$ does not in general exclude a torus in a finite \'etale cover, since the deck group can eliminate invariant one-forms. Under the hypotheses of that proposition, $\chi(\mathcal O)=n+1\neq0$ excludes a positive-dimensional torus factor by multiplicativity of the holomorphic Euler characteristic under finite \'etale covers and on products, as above (see also \cite[Remark 2]{Schwald}).
\end{rem}

\newpage
\bibliographystyle{alpha}
\bibliography{derived-fivefold}
\end{document}